\documentclass[leqno]{amsart}
\usepackage{amssymb}
\usepackage{amsfonts}
\usepackage{amsfonts}
\usepackage{amsfonts}
\usepackage{amsfonts}
\usepackage{mathrsfs}
\usepackage{amssymb}
\usepackage{txfonts}
\usepackage{amsfonts}
\usepackage{amssymb,amsmath}

\usepackage{graphicx}
\usepackage{subfigure}

\usepackage{graphicx,epic,eepic}

\allowdisplaybreaks

\begin{document}

\title[Fractional p-Choquad equations without (AR) condition]
{Existence of solutions for fractional $p-$Kirchhoff type equations with a generalized Choquard nonlinearities}

\author[W. Chen]{Wenjing Chen}
\address[W. Chen]{\noindent School of Mathematics and Statistics, Southwest University,
Chongqing 400715, People's Republic of China.}\email{wjchen@swu.edu.cn
}

 \maketitle

\maketitle
\numberwithin{equation}{section}
\newtheorem{theorem}{Theorem}[section]
\newtheorem{lemma}[theorem]{Lemma}
\newtheorem{definition}[theorem]{Definition}
\newtheorem{proposition}[theorem]{Proposition}
\newtheorem{remark}[theorem]{Remark}
\allowdisplaybreaks

\maketitle

\noindent {\bf Abstract}: In this article, we establish the existence of solutions to the fractional $p-$Kirchhoff type equations with a generalized Choquard nonlinearities without assuming the Ambrosetti-Rabinowitz condition.

\vspace{3mm} \noindent {\bf Keywords}: fractional p-Kirchhoff type equations; Choquard equation; without the (AR) condition.

\vspace{3mm} \noindent {\bf MR(2010) Subject Classification}:  	35J20, 35J60,  	47G20
\vspace{3mm}

\maketitle

\section{Introduction and statement of main result}

In this work,  we consider the following fractional $p-$Laplacian generalized Choquard equation
\begin{equation}\label{e1.1}
M(\|u\|^p_W)\Big[ (-\Delta)^s_pu + V(x)|u|^{p-2}u\Big]
= \lambda (\mathcal{I}_\mu*F(u))f(u),\quad \mbox{in}\ \ \mathbb{R}^N\\
\end{equation}
where $1<ps<N$, $M:\mathbb{R}^+_0\to\mathbb{R}^+$ is a  Kirchhoff
function,
\begin{equation}\label{e1.2}
\|u\|_{W}=\Big([u]_{s,p}^p+\int_{\mathbb{R}^N}V(x)|u|^pdx\Big)^{1/p}\quad \mbox{with}\ \ [u]_{s,p} =\Big(\iint_{\mathbb{R}^{2N}}\frac{|u(x)-u(y)|^p}{|x-y|^{N+ps}}\,dx\,dy
\Big)^{1/p},
\end{equation}
the potential function $V:\mathbb{R}^N\to\mathbb{R}^+$ is continuous, $f\in C(\mathbb{R},\mathbb{R})$ and $F\in C(\mathbb{R},\mathbb{R})$
with $F(u)=\int_0^uf(t)dt$,
here $\mathcal{I}_\mu(x)=|x|^{-\mu}$ is the Riesz potential of order $\mu\in(0,ps)$, and $(-\Delta )^s_p$ is the
fractional $p-$Laplacian operator which, up to a normalization constant, is
defined as
\begin{equation*}
(-\Delta)_p^s\varphi(x)=2 \lim_{\varepsilon\to
0^+}\int_{\mathbb{R}^N\setminus
B_\varepsilon(x)}\frac{|\varphi(x)-\varphi(y)|^{p-2}
(\varphi(x)-\varphi(y))}{|x-y|^{N+ps}}\,dy,\quad
x\in\mathbb{R}^{N},
\end{equation*}
along functions $\varphi\in C_0^\infty(\mathbb{R}^N)$, where
$B_\varepsilon(x)$ denotes the ball of $\mathbb{R}^N$ centered at
$x\in\mathbb{R}^N$ and radius $\varepsilon>0$.

On the one hand, this paper is motivated by some works that has been focused on the study of Kirchhoff type problems.
Fiscella and Valdinoci \cite{fiscellaValdinoci} first proposed a stationary fractional Kirchoff variational model as follows
\begin{eqnarray}\label{frac1cr2}
\qquad \left\{ \arraycolsep=1.5pt
   \begin{array}{ll}
 M\left(\iint_{\mathbb{R}^{2N}}\frac{|u(x)-u(y)|^2}{|x-y|^{N+2s}}dxdy\right) (-\Delta)^s u(x)=\lambda   f(x,u)+  |u|^{2^\ast-2}u \ \ \ &
{\rm in}\ \Omega,\\[2mm]
u=0 \ \ \quad & {\rm in}\ \mathbb{R}^N\backslash\Omega,
\end{array}
\right.
\end{eqnarray}
where $\Omega\subset\mathbb{R}^N$ is an open bounded set, $2^\ast=\frac{2N}{N-2s}$, $N>2s$ with $s\in(0,1)$. $M$ and $f$ are two continuous functions under some suitable assumptions. In \cite{fiscellaValdinoci}, the authors first provided a detail discussion about the physical meaning underlying the fractional Kirchhoff problems and their applications.
They supposed that $M:\mathbb{R}^+\to \mathbb{R}^+$ is an increasing and continuous function, and
there exists $m_0>0$ such that $M(t)\geq m_0= M(0)$ for all $t\in\mathbb{R}^+$.
Based on the truncated skill and the mountain pass theorem, they obtained the existence of a non-negative solution to problem (\ref{frac1cr2}) for any $\lambda>\lambda^\ast>0$, where $\lambda^\ast$ is an appropriate threshold.
Autuori {\em et al.} \cite{afp} established the existence and the asymptotic behavior of non-negative solutions to problem (\ref{frac1cr2}) under different assumptions on $M$,  the Kirchhoff function $M$ can be zero at zero, that is, the problem is degenerate case.

Moreover, there is a lot of literature concerning the existence and multiplicity of solutions for the fractional $p-$Laplacian Kirchhoff type problems.
Xiang {\em et al.} in \cite{XZF1} investigated the existence of solutions for Kirchhoff type problems involving the fractional $p-$Laplacian by variational methods, where the nonlinearity is subcritical and the Kirchhoff function is non-degenerate.
Combining the mountain pass theorem with Ekeland variational principle, Xiang {\em et al.} in \cite{XZF} established the existence of two solutions for a degenerate fractional $p-$Laplacian Kirchhoff equation in $\mathbb{R}^N$ with concave-convex nonlinearity.
By the same methods as in \cite{XZF}, Pucci {\em et al.} in \cite{PXZ} obtained the existence of two solutions for a nonhomogenous  Schr\"{o}dinger-Kirchhoff type equation involving the fractional $p-$Laplacian in $\mathbb{R}^N$ on a nondegenerate situation.
Furthermore, nonexistence and multiplicity of solutions for a nonhomogeneous fractional $p-$Kirchhoff type problem involving critical exponent in $\mathbb{R}^N$  were studied in \cite{XZZ}.
The existence of infinitely many solutions was proved in \cite{PXZ2,XMTZ} by using Krasnoselskii's genus theory under degenerate frameworks.
Recently, Song and Shi considered the existence of infinitely many solutions for degenerate $p-$fractional Kirchhoff equations with critical Sobolev-Hardy nonlinearities in \cite{SS1,SS2}.

On the other hand, there are some results about the Choquard equation,
consider the following Choquard or nonlinear
Schr\"odinger-Newton equation
\begin{align}\label{eq1.01}
-\Delta u+V(x)u=(\mathcal{I}_\mu*u^2)u+\lambda f(x,u)\quad\text{in
}\mathbb{R}^N,
\end{align}
which was elaborated by  Pekar \cite{Pekar} in the framework of
quantum mechanics.
The first investigation for the existence and symmetry of solutions to
\eqref{eq1.01} went back to the works of Lieb \cite{Lieb}. 
Equations of type \eqref{eq1.01} have
been extensively studied, see e.g. \cite{AFY1,MS1,MS2} and references therein.
Moroz and van Schaftingen in \cite{MS2} considered the existence of ground-states for a generalized Choquard equation.
The existence, multiplicity and concentration of solutions for a generalized quasilinear Choquard equation were studied by
Alves and Yang in \cite{AY,AY2}. We refer to \cite{MS4} for a good survey of the Choquard equation.

In the setting of the fractional Choquard equations,
\begin{align}\label{eq1.02}
(-\Delta)^s u+V(x) u=(\mathcal{I}_\mu*F(u))f(u) \quad\text{in }\mathbb{R}^N,
\end{align}
Wu  \cite{DW} investigated existence and stability of solutions
to (\ref{eq1.02}) with $f(u)=u$ and $\mu \in (N-2s, N)$.
Subsequently, D'Avenia and  Squassina in \cite{PDSS} studied the existence, regularity and asymptotic behavior of solutions to \eqref{eq1.02} with $f(u)=u^p$ and $V(x)\equiv const$. In particular, they claimed the nonexistence
of solutions as $q\in (\frac{2N-\mu}{N}, \frac{2N-\mu}{N-2s})$.
If $V(x)=1$ and $f$ satisfies Berestycki-Lions type assumptions, the existence of ground state solutions for  a fractional Choquard equation has been established in \cite{SGY}. Very recently, Ambrosio studied the concentration phenomena of solutions for a fractional Choquard equation with mangetic field in \cite{Amb}.

Recently, Belchior et al. in \cite{BBMP1} applied  the mountain pass theorem without PS condition and a characterization of
the infimum more suitable to the Nehari manifold naturally attached to the problem to
study the existence of ground state,
regularity and polynomial decay for the following fractional Choquard equation
\begin{align}\label{eqo1}
(-\Delta)^s_p u+A|u|^{p-2}u= (\mathcal{I}_\mu\ast F(u))f(u)\quad \mbox{in}\ \mathbb{R}^N,
\end{align}
where $A$ is a positive constant, $f$ is a $C^1$ positive function on $(0,\infty)$, $\lim_{t\to0}\frac{|f(t)|}{t^{p-1}}=0$,
$\lim_{t\to\infty}\frac{ f(t) }{t^{q-1}}=0$ for some $p<q<\frac{(2N-\mu)p}{2(N-ps)}$, and
\begin{align}\label{arb}
f'(t)t^2-(p-1)f(t)t>0\ \ \mbox{for\ all}\ t>0.
\end{align}
An example of a function $f$ satisfying these hypotheses is given by $f(t)=|t|^{q_1-1}t^++|t|^{q_2-1}t^+$, where $p<q_1<q_2<\frac{(N-\mu)p}{N-ps}$ and $t^+=\max\{t,0\}$. From (\ref{arb}), $f$ satisfies the
Ambrosetti-Rabinowitz condition ((AR) for short):
\begin{align}\label{ar}
p F(t)<tf(t)\ \ \mbox{for\ all}\ t>0,
\end{align}
and the function $\frac{f(t)}{t^{p-1}}$ is increasing.
It is well known that the $(AR)-$condition is quite natural and important not only to ensure that an Euler-Lagrangian functional has the mountain pass geometry structure, but also to ensure that the Palais-Smale sequence of the functional are bounded. However, there are many functions which are superlinear at infinity, but do not satisfy the $(AR)-$condition, for example, the function $f(t)=|t|^{p-2}t\log(1+|t|)$.
Thus, many researchers have tried to drop the $(AR)-$condition for elliptic equations involving the $p-$Laplacian, see \cite{fangliu,lkbp,LiYang,MS} and references therein.

In particular, Lee et al. in \cite{lkbp} considered the existence of nontrivial weak solutions for the quasilinear Choquard equation with
the nonlinearity $f$ does not satisfy the $(AR)-$condition.

Motivated by the above results, in the present paper, we are interested in the existence of solutions for the fractional $p-$Kirchhoff type equation (\ref{e1.1}) with a generalized Choquard nonlinearities without assuming the Ambrosetti-Rabinowitz condition.
We first give the following assumptions on the potential function $V$ and the Kirchhoff function $M$.

\begin{itemize}
\item[($V$)]$V:\mathbb{R}^N\to\mathbb{R}^+$
is a continuous function and there exists $V_0>0$ such that
$\inf_{\mathbb{R}^N} V \geq V_0$.
\end{itemize}

\begin{itemize}
\item[($M_1$)] $M:\mathbb{R}^+_0\to\mathbb{R}^+$
is a continuous function and there exists $m_0>0$ such that
$\inf_{t\geq 0} M(t) =m_0$.
\item[($M_2$)] There exists $\theta\in [1,\frac{2N-\mu}{N})$ 
such that
$$
M(t)t\leq \theta \mathscr{M}(t),\quad \forall\ t\geq 0,
$$
where $\mathscr{M}(t)=\int_0^t M(\tau)d\tau$.
\end{itemize}
A typical example is $M(t)=m_0+bt^{\theta-1}$, where $b\geq0$, $t\geq0$.

Moreover, we impose the following assumption on the nonlinearity
$f: \mathbb{R}\to\mathbb{R}$
that
\begin{itemize}
\item[($F_1$)]  $F\in C^1(\mathbb{R},\mathbb{R})$.

\item[($F_2$)]  There exist a constant $c_0>0$ and $p <q_1\leq q_2<\frac{(N-\mu)p}{N-ps}$ such that for all $t\in\mathbb{R}$,
    \[
    |f(t)|\leq c_0(|t|^{q_1-1}+|t|^{q_2-1}).
    \]

\item[($F_3$)]  $\lim\limits_{|u(x)|\to\infty}\frac{F(u(x))}{|u(x)|^{p\theta}}=\infty$ uniformly for $x\in\mathbb{R}^N$.

\item[($F_4$)]  There exist $c_1\geq0$, $r_0\geq0$ and $\kappa>\frac{N}{ps}$ such that
\[
|F(t)|^\kappa\leq c_1|t|^{\kappa p}\mathscr{F}(t)
\]
for all $t\in\mathbb{R}$ and $|t|\geq r_0$, where $\mathscr{F}(t)=\frac{1}{p\theta}f(t)t-\frac{1}{2}F(t)\geq0$.
\end{itemize}

The main result is as follows.

\begin{theorem}\label{thm1.1}
Let $0<\mu<ps<N$, and $(V)$, $(M_1)-(M_2)$ and $(F_1)-(F_4)$ hold. Then problem \eqref{e1.1} has a nontrivial weak solution for any $\lambda>0$.
\end{theorem}

The paper is organized as follows. In Section~\ref{sec preliminaries}, we give some definitions and preliminaries.
Section \ref{main} is devoted to prove Theorem \ref{thm1.1}, we obtain the existence of solution to problem \eqref{e1.1} by the mountain pass theorem.

\section{Preliminaries}\label{sec preliminaries}

We introduce some useful notations. The fractional Sobolev
space $W^{s,p}(\mathbb{R}^N)$ is defined by
\[
W^{s,p}(\mathbb{R}^N)=\left\{u\in L^p(\mathbb{R}^N)\ :\ [u]_{s,p}<\infty\right\},
\]
where $[u]_{s,p}$ denotes the Gagliardo norm defined by
\[
[u]_{s,p}=\left(\int\int_{\mathbb{R}^{2N}}\frac{|u(x)-u(y)|^p}{|x-y|^{N+ps}}dxdy\right)^{1/p},
\]
and $W^{s,p}(\mathbb{R}^N)$ is equipped with the norm
\[
\|u\|_{W^{s,p}(\mathbb{R}^N)}=\left(\|u\|_{p}^p+[u]_{s,p}^p\right)^{1/p},
\]
where and hereafter we denote by
$\|\cdot\|_q$ the norm of Lebesgue space $L^q(\mathbb{R}^N)$.
As it is well-known,  $W^{s,p}(\mathbb{R}^N)=(W^{s,p}(\mathbb{R}^N),\|u\|_{W^{s,p}(\mathbb{R}^N)})$ is a uniformly convex Banach
space.
Let $L^p(\mathbb{R}^N,V)$ denote the Lebesgue space of
real valued functions, with $V(x)|u|^p\in L^1({\mathbb{R}}^N),$
equipped with norm
$$
\|u\|_{p,V}=\Big(\int_{\mathbb{R}^N}V(x)|u|^p
\,dx\Big)^{1/p}\quad \text{for all }u\in L^p(\mathbb{R}^N,V).
$$
Let $W_V^{s,p}(\mathbb{R}^N)$ denote the completion of $C_0^{\infty}(\mathbb{R}^N)$, with respect to the norm
\[
\|u\|_{W}=\left([u]_{s,p}^p+\|u\|_{p,V}^p\right)^{1/p}.
\]
The embedding $W^{s,p}_V(\mathbb{R}^N)\hookrightarrow
L^{\nu}(\mathbb{R}^N)$ is continuous for any $\nu\in [p,\frac{Np}{N-ps}]$ by
\cite[Theorem 6.7]{r28}, namely there exists a positive constant
$C_\nu$ such that
\begin{align}\label{sobem}
\|u\|_{ \nu }\leq C_\nu \|u\|_W\quad\text{for all }
u\in W^{s,p}_V(\mathbb{R}^N).
\end{align}

Next, we recall the Hardy-Littlewood-Sobolev
inequality.
\begin{theorem} \cite[Theorem 4.3]{LL}
Assume that $1<r$, $t<\infty$, $0<\mu<N$ and
$$
\frac{1}{r}+\frac{1}{t}+\frac{\mu}{N}=2.
$$
Then there exists $C(N,\mu,r,t)>0$ such that
\begin{align*}
\iint_{\mathbb R^{2N}}\frac{|g(x)|\cdot|h(y)|}{|x-y|^\mu}\,dx\,dy\leq
C(N,\mu,r,t)\|g\|_r\|h\|_t
\end{align*}
for all $g\in L^r(\mathbb{R}^N)$ and $h\in L^t(\mathbb{R}^N)$.
\end{theorem}

In particular, $F(t)=|t|^{q_1}$ for some $q_1>0$, by the Hardy-Littlewood-Sobolev inequality, the
integral
\begin{align*}
\iint_{\mathbb{R}^{2N}}\frac{ F(u(x))F(u(y))}{|x-y|^\mu}\,dx\,dy
\end{align*}
is well defined if  $F\in L^t(\mathbb{R}^N)$ for some $t>1$
satisfying
\[
\frac{2}{t}+\frac{\mu}{N}=2, \quad\text{that is }t=\frac{2N}{2N-\mu}.
\]
Hence, by the fractional Sobolev embedding theorem, if
$u\in W_{V}^{s,p}(\mathbb{R}^N)$, we must require that $tq_1 \in[p,
\frac{Np}{N-ps}]$. Thus, for the subcritical case, we must assume
\begin{align*}
\tilde p_{\mu,s}=\frac{(N-\mu/2)p}{N}< q_1\leq q_2<
\frac{(N-\mu/2)p}{N-ps}=p_{\mu,s}^*.
\end{align*}
Hence, $\tilde p_{\mu,s}$ is called the lower critical exponent and
$p_{\mu,s}^*$ is said to be the upper critical exponent in the sense
of the Hardy-Littlewood-Sobolev inequality.

Equation \eqref{e1.1} has a variational structure and its associated energy functional   $\mathcal
J_\lambda: W_{V}^{s,p}(\mathbb{R}^N) \to\mathbb R$ is defined by
\[
\mathcal J_\lambda(u)
=\Phi(u)-  \lambda \Psi(u).
\]
with
\[
\Phi(u):=\frac{1}{p}\mathscr{M}(\|u\|_W^p),\quad \mbox{and}\ \ \
\Psi(u):=  \frac{1}{2}
\iint_{\mathbb{R}^{2N}}\frac{ F(u(x))F(u(y))}{|x-y|^\mu}\,dx\,dy.
\]
Under  the assumption ($F_2$),  $\mathcal
J_\lambda$ is of class $C^1(W_{V}^{s,p}(\mathbb{R}^N),\mathbb{R})$.
We say that $u\in W_V^{s,p}(\mathbb{R}^N)$ is a weak
solution of problem \eqref{e1.1}, if
\begin{gather*}
\begin{aligned}
 M(\|u\|_W^p)\Big[\langle u,\varphi\rangle_{s,p}+\int_{\mathbb{R}^N}V|u|^{p-2}u\varphi\,dx \Big]= \lambda\int_{\mathbb{R}^N} (\mathcal{I}_\mu*F(u))f(u)\varphi\,dx,
\end{aligned}
\end{gather*}
for all $\varphi\in W_V^{s,p}(\mathbb{R}^N)$, where
\[
\langle u,\varphi\rangle_{s,p}
=\iint_{\mathbb{R}^{2N}}\frac{\big[|u(x)-u(y)|^{p-2}(u(x)-u(y))\big]
\cdot\big[\varphi(x)-\varphi(y)\big]}{|x-y|^{N+ps}} \,dx\,dy.
\]
Clearly, the critical points of $\mathcal J_\lambda$ are exactly the
weak solutions of problem \eqref{e1.1}.

\begin{lemma}\cite[Lemma 2]{PXZ},\label{lema1}
Let $(V)$ and $(M_1)$ hold. Then $\Phi$ is of class $C^1(W^{s,p}_V(\mathbb{R}^N),\mathbb{R})$ and
\begin{align*}
\langle \Phi'(u),\varphi\rangle=&M(\|u\|_W^p)\Big[\iint_{\mathbb{R}^{2N}}\frac{|u(x)-u(y)|^{p-2}(u(x)-u(y))(\varphi(x)-\varphi(y))}{|x-y|^{N+ps}}dxdy\\
&\qquad\qquad+\int_{\mathbb{R}^N}V(x)|u(x)|^{p-2}u(x)\varphi(x)dx\Big],
\end{align*}
for all $u,\varphi\in W^{s,p}_V(\mathbb{R}^N)$. Moreover, $\Phi$ is weakly lower semi-continuous in $W^{s,p}_V(\mathbb{R}^N)$.
\end{lemma}

The next result is stated in \cite{AY}.

\begin{lemma}\label{estf}
Assume ($F_2$) holds, then there exists $K>0$ such that
\begin{align}\label{asd1}
|\mathcal{I}_\mu\ast F(v)|\leq K\quad \mbox{for}\ v\in W^{s,p}_V(\mathbb{R}^N).
\end{align}
\end{lemma}


\begin{lemma}\label{lema2a}
Let $(V)$ and $(F_1)-(F_2)$ hold. Then $\Psi$ and $\Psi'$ are weakly strongly continuous on $W^{s,p}_V(\mathbb{R}^N)$.
\end{lemma}

\begin{proof}
Let $\{u_n\}$ be a sequence in $W^{s,p}_V(\mathbb{R}^N)$ such that $u_n\rightharpoonup u$ in $W^{s,p}_V(\mathbb{R}^N)$ as $n\to\infty$. Then $\{u_n\}$ is bounded in $W^{s,p}_V(\mathbb{R}^N)$, and then there exists a subsequence denoted by itself, such that
$$
u_n\to u\quad\mbox{in}\ L^{q_1}(\mathbb{R}^N)\cap L^{q_2}(\mathbb{R}^N),\qquad \mbox{and}\ \ \
u_n\to u\quad\mbox{a.e. in}\ \mathbb{R}^N\ \ \mbox{as}\ n\to\infty,
$$
and by \cite[Theorem IV-9]{bre} there exists $\ell\in L^{q_1}(\mathbb{R}^N)\cap L^{q_2}(\mathbb{R}^N)$ such that
$$
|u_n(x)|\leq \ell(x)\ \mbox{a.e.\ in}\ \ \mathbb{R}^N.
$$
First, we show that $\Psi$ is weakly strongly continuous on $W^{s,p}_V(\mathbb{R}^N)$. Since $F\in C^1(\mathbb{R},\mathbb{R})$, we see that $F(u_n)\to F(u)$ as
$n\to\infty$ for almost all $x\in\mathbb{R}^N$, and so $(\mathcal{I}_\mu\ast F(u_n))F(u_n)\to (\mathcal{I}_\mu\ast F(u))F(u)$ as $n\to\infty$  for almost all $x\in\mathbb{R}^N$.
From Lemma \ref{estf} and $(F_2)$, we have
\begin{align*}
|(\mathcal{I}_\mu\ast F(u_n))F(u_n)|
\leq Kc_0 \Big(\frac{|u_n(x)|^{q_1}}{q_1}+\frac{|u_n(x)|^{q_2}}{q_2}\Big)\in L^1(\mathbb{R}^N).
\end{align*}
By Lebesgue dominated convergence theorem, we get
$$
\int_{\mathbb{R}^N}(\mathcal{I}_\mu\ast F(u_n))F(u_n)dx\to \int_{\mathbb{R}^N}(\mathcal{I}_\mu\ast F(u))F(u)dx\quad \mbox{as}\ n\to\infty,
$$
which implies that $\Psi(u_n)\to\Psi(u)$ as $n\to\infty$. Thus $\Psi$ is weakly strongly continuous on $W^{s,p}_V(\mathbb{R}^N)$.

We next prove that $\Psi'$ is weakly strongly continuous on $W^{s,p}_V(\mathbb{R}^N)$. Since $u_n(x)\to u(x)$ as $n\to\infty$ for almost all $x\in\mathbb{R}^N$,
$f(u_n)\to f(u)$ for almost all $x\in\mathbb{R}^N$ as $n\to\infty$. Then
$$
(\mathcal{I}_\mu\ast F(u_n))f(u_n) \to  (\mathcal{I}_\mu\ast F(u))f(u)\quad \mbox{a.e.\ in}\ \mathbb{R}^N,\ \  \mbox{as}\ n\to\infty.
$$
By $(F_2)$ and H\"{o}lder inequality, we have that for any $\varphi\in W^{s,p}_V(\mathbb{R}^N)$,
\begin{align*}
&\int_{\mathbb{R}^N}|(\mathcal{I}_\mu\ast F(u_n))f(u_n) \varphi(x)|dx\\
\leq&c_0K\int_{\mathbb{R}^N}| (|u_n|^{q_1-1}+|u_n|^{q_2-1}) \varphi(x)|dx\\
\leq&c_0K  \Big(\|u_n\|_{q_1}^{q_1-1}\|\varphi\|_{q_1}
+\|u_n\|_{q_2}^{q_2-1}\|\varphi\|_{q_2}\Big)   \\
\leq&c_0K  \Big(C_{q_1}\|\ell(x)\|_{q_1}^{q_1-1}
+C_{q_2}\|\ell(x)\|_{q_2}^{q_2-1}\Big)\|\varphi\|_W.
\end{align*}
Then by Lebesgue dominated convergence theorem, we obtain
\begin{align*}
&\|\Psi'(u_n)-\Psi'(u)\|_{\big(W_V^{s,p}(\mathbb{R}^N)\big)'}\\
=&\sup\limits_{\|\varphi\|_{W_V^{s,p}(\mathbb{R}^N)}=1}|\langle \Psi'(u_n)-\Psi'(u), \varphi\rangle|\\
=&\sup\limits_{\|\varphi\|_{W_V^{s,p}(\mathbb{R}^N)}=1}\int_{\mathbb{R}^N}|(\mathcal{I}_\mu\ast F(u_n))f(u_n) \varphi(x)-(\mathcal{I}_\mu\ast F(u))f(u) \varphi(x)|dx\\
\to& 0\quad \mbox{as}\ n\to\infty.
\end{align*}
Therefore, we get that $\Psi'(u_n)\to\Psi'(u)$ in $\big(W_V^{s,p}(\mathbb{R}^N)\big)'$ as $n\to\infty$. This completes the proof.
\end{proof}

\section{Proof of the main result}\label{main}

In this section, we will prove our main result. First, we introduce the following definition.

\begin{definition}\label{ccc}
For $c\in\mathbb{R}$, we say that $\mathcal{J}_\lambda$ satisfies the $(C)_c$ condition if for any sequence $\{u_n\}\subset W^{s,p}_V(\mathbb{R}^N)$ with
$$
\mathcal{J}_\lambda(u_n)\to c,\quad \|\mathcal{J}_\lambda'(u_n)\|(1+\|u_n\|_W)\to 0,
$$
there is a subsequence $\{u_n\}$ such that $\{u_n\}$ converges strongly in $W^{s,p}_V(\mathbb{R}^N)$.
\end{definition}

We will use the following  mountain pass theorem to prove  our result.
\begin{lemma}[Theorem 1 in \cite{CM}]\label{cmainth}
Let $E$ be a real Banach space, $I\in C^1(E,\mathbb{R})$
satisfies the $(C)_c$ condition for any $c\in\mathbb{R}$, and

(i) There are constants $\rho,\alpha>0$ such that $I|_{\partial B_\rho}\geq\alpha$.

(ii) There is an $e\in E\backslash B_\rho$ such that $I(e)\leq 0$.\\
\noindent Then,
$$
c=\inf\limits_{\gamma\in\Gamma}\max\limits_{0\leq t\leq1}I(\gamma(t))\geq \alpha
$$
is a critical value of $I$, where
$$
\Gamma=\{\gamma\in C([0,1],E):\gamma(0)=0,\gamma(1)=e\}.
$$
\end{lemma}

We first show that the energy functional $\mathcal{J}_\lambda$ satisfies the geometric structure.

\begin{lemma}\label{mage}
Assume that $(V)$, $(M_1)-(M_2)$ and $(F_1)-(F_3)$ hold. Then

(i) There exists $\alpha,\rho>0$ such that $\mathcal{J}_\lambda(u)\geq\alpha$ for all $u\in W^{s,p}_V(\mathbb{R}^N)$ with $\|u\|_W=\rho$.

(ii) $\mathcal{J}_\lambda(u)$ is unbounded from below on $W^{s,p}_V(\mathbb{R}^N)$.
\end{lemma}

\begin{proof}
$(i)$ From Lemma \ref{estf} and $(M_1)-(M_2)$, ($F_2$), we have
\begin{align*}
\mathcal J_\lambda(u)
=&\frac{1}{p}\mathscr{M}(\|u\|_W^p)-  \frac{\lambda }{2}\iint_{\mathbb{R}^{2N}}\frac{ F(u(x))F(u(y))}{|x-y|^\mu}\,dx\,dy\\
\geq &\frac{1}{p\theta}M(\|u\|_W^p)\|u\|_W^p-  \frac{\lambda c_0K }{2}\int_{\mathbb{R}^{N}}\Big(\frac{|u|^{q_1}}{q_1}+\frac{|u|^{q_2}}{q_2}\Big)\,dx\\
\geq &\left[\frac{m_0}{p\theta} -  \frac{\lambda c_0K }{2} \Big(C_{q_1}^{q_1}\|u\|_W^{q_1-p}+C_{q_2}^{q_2}\|u\|_W^{q_2-p}\Big)\right]\|u\|_W^p.
\end{align*}
Since $q_2\geq q_1>p$, the claim follows if we choose $\rho$ small enough.

$(ii)$ From $(M_2)$, we have
\begin{align}\label{equ0}
\mathscr{M}(t)\leq\mathscr{M}(1)t^\theta\quad \mbox{for\ all}\  t\geq 1.
\end{align}
By the assumption $(F_3)$, we can take that $t_0$ such that $F(t_0)\neq 0$, we find
$$
\int_{\mathbb{R}^N}(\mathcal{I}_\mu\ast F(t_0\chi_{B_1}))F(t_0\chi_{B_1})dx=F(t_0)^2\int_{B_1}\int_{B_1}\mathcal{I}_\mu(x-y)dxdy>0,
$$
where $B_r$ denotes the open ball centered at the origin with radius $r$ and $\chi_{B_1}$ denotes the standard indicator function of set $B_1$. By the density theorem, there will be
$v_0\in W^{s,p}_V(\mathbb{R}^N)$ with
$$
\int_{\mathbb{R}^N}(\mathcal{I}_\mu\ast F(v_0))F(v_0)dx>0.
$$
Define the function $v_t(x)=v_0(\frac{x}{t})$, then
\begin{align*}
\mathcal J_\lambda(v_t)
=&\frac{1}{p}\mathscr{M}(\|v_t\|_W^p)-  \frac{\lambda }{2}\iint_{\mathbb{R}^{2N}}\frac{ F(v_t(x))F(v_t(y))}{|x-y|^\mu}\,dx\,dy\\
\leq &\frac{1}{p}\mathscr{M}(1)\|v_t\|_W^{p\theta}  -  \frac{\lambda }{2}\iint_{\mathbb{R}^{2N}}\frac{ F(v_t(x))F(v_t(y))}{|x-y|^\mu}\,dx\,dy\\
=&\frac{1}{p}\mathscr{M}(1)\left[t^{N-ps}\|v_0\|_W^{p}+t^N\int_{\mathbb{R}^N}V(tx)|v_0|^p dx\right]^\theta
-  t^{2N-\mu}\frac{\lambda }{2}\iint_{\mathbb{R}^{2N}}\frac{ F(v_0(x))F(v_0(y))}{|x-y|^\mu}\,dx\,dy,
\end{align*}
for sufficiently large $t$. Therefore, we have that $\mathcal J_\lambda(v_t)\to-\infty$ as $t\to\infty$ since $1\leq \theta<\frac{2N-\mu}{N}$ gives that $2N-\mu>N\theta>(N-ps)\theta$. Hence we obtain that
the functional $\mathcal{J}_\lambda$ is unbounded from below.
\end{proof}

\begin{lemma}\label{mage1}
Assume that $(V)$, $(M_1)-(M_2)$ and $(F_1)-(F_4)$ hold. Then $(C)_c-$sequence of $\mathcal{J}_\lambda$ is bounded for any $\lambda>0$.
\end{lemma}

\begin{proof}
Suppose that  $\{u_n\}\subset W^{s,p}_V(\mathbb{R}^N)$ is a $(C)_c-$sequence for $\mathcal{J}_\lambda(u)$, that is.
$$
\mathcal{J}_\lambda(u_n)\to c,\qquad \|\mathcal{J}'_\lambda(u_n)\|_W(1+\|u_n\|_W)\to0,
$$
which shows that
\begin{align}\label{equ1}
c=\mathcal{J}_\lambda(u_n)+o(1),\qquad \langle\mathcal{J}'_\lambda(u_n),u_n\rangle=o(1)
\end{align}
where $o(1)\to0$ as $n\to\infty$. We now prove that $\{u_n\}$ is bounded in $W^{s,p}_V(\mathbb{R}^N)$. We argue by contradiction. Suppose that the sequence $\{u_n\}$ is unbounded in $W^{s,p}_V(\mathbb{R}^N)$, then we may assume that
\begin{align}\label{equ2}
\|u_n\|_W\to\infty,\ \ \mbox{as}\ n\to\infty.
\end{align}
Let $\omega_n(x)=\frac{u_n}{\|u_n\|_W}$, then $\omega_n\in W^{s,p}_V(\mathbb{R}^N)$ with $\|\omega_n\|_W=1$. Hence, up to a subsequence, still denoted by itself, there exists a function $\omega\in W^{s,p}_V(\mathbb{R}^N)$ such that
\begin{align}\label{equ3}
\omega_n(x)\to\omega(x)\quad \mbox{a.e.\ in}\ \mathbb{R}^N,\qquad \mbox{and}\ \
\omega_n(x)\to\omega(x)\quad \mbox{a.e.\ in}\ L^r(\mathbb{R}^N)
\end{align}
as $n\to\infty$,  for $p\leq r<\frac{Np}{N-ps}$.

Let $\Omega_1=\{x\in\mathbb{R}^N: \omega(x)\neq0\}$, then
\[
\lim\limits_{n\to\infty}\omega_n(x)=\lim\limits_{n\to\infty}\frac{u_n(x)}{\|u_n\|_W}=\omega(x)\neq 0\ \ \mbox{in}\ \Omega_1,
\]
and (\ref{equ2}) implies that
\begin{align}\label{equ4}
|u_n|\to\infty\ \ \mbox{a.e.\ in}\ \Omega_1.
\end{align}
So from the assumption $(F_3)$ and Lemma \ref{estf}, we have
\begin{align}\label{equ4a}
 \lim\limits_{n\to\infty}\frac{(\mathcal{I}_\mu\ast F(u_n(x)))F(u_n(x))}{|u_n(x)|^{p\theta}}|\omega_n(x)|^{p\theta}=\infty,\ \ \mbox{for\ a.e.}\ x\in\Omega_1.
\end{align}

Moreover, by ($F_3$), there exists $t_0>0$ such that
$$
 \frac{F(t)}{|t|^{p\theta}}>1,
$$
for all $|t|>t_0$. Since $F$ is continuous, then there exists $\mathcal{C}>0$ such that
$|F(t)|\leq \mathcal{C}$ for all $t\in[-t_0,t_0]$. Thus, we see that there is a constant $C_0$ such that for any
$t\in\mathbb{R}$, we have
$F(t)\geq C_0$, which show that  there is a constant $C$ such that
$$
\frac{(\mathcal{I}_\mu\ast F(u_n))F(u_n)- C}{\|u_n\|_W^{p\theta}}\geq0.
$$
This means that
\begin{align}\label{equ5}
\frac{(\mathcal{I}_\mu\ast F(u_n))F(u_n(x))}{|u_n(x)|^{p\theta}}|\omega_n(x)|^{p\theta}-\frac{ C}{\|u_n\|_W^{p\theta}}\geq0.
\end{align}
By (\ref{equ1}) we have that
\begin{align}\label{equ5b}
c=\mathcal{J}_\lambda(u_n)+o(1)
=\frac{1}{p}\mathscr{M}(\|u_n\|_W^p)-  \frac{\lambda }{2}\int_{\mathbb{R}^{N}}(\mathcal{I}_\mu\ast F(u_n))F(u_n)dx +o(1).
\end{align}
Using this and $(M_1)-(M_2)$, we find
\begin{align}\label{equ6}
\frac{1}{2}\int_{\mathbb{R}^{N}}(\mathcal{I}_\mu\ast F(u_n))F(u_n)dx
=&\frac{1}{p\lambda}\mathscr{M}(\|u_n\|_W^p)-  \frac{c}{\lambda}+\frac{o(1)}{\lambda}\nonumber\\
\geq &\frac{m_0}{p\theta\lambda} \|u_n\|_W^{p}-  \frac{c}{\lambda}+\frac{o(1)}{\lambda}\nonumber\\
\to&\infty,\quad\mbox{as}\ \ n\to\infty.
\end{align}
We claim that $\mbox{meas}(\Omega_1)=0$. Indeed, if $\mbox{meas}(\Omega_1)\neq 0$.
from (\ref{equ0}), (\ref{equ4a}), (\ref{equ5}), (\ref{equ5b}) and Fatou's lemma, we have
\begin{align}\label{equ6a}
+\infty=&
\int_{\Omega_1}\liminf\limits_{n\to\infty}\frac{(\mathcal{I}_\mu\ast F(u_n(x)))F(u_n(x))}{|u_n(x)|^{p\theta}}|\omega_n(x)|^{p\theta}dx
-  \int_{\Omega_1}\limsup\limits_{n\to\infty}\frac{ C }{ \|u_n\|_W^{p\theta} }dx\nonumber\\
\leq&\int_{\Omega_1}\liminf\limits_{n\to\infty}\left(\frac{(\mathcal{I}_\mu\ast F(u_n(x)))F(u_n(x))}{|u_n(x)|^{p\theta}}|\omega_n(x)|^{p\theta}
-  \frac{ C }{ \|u_n\|_W^{p\theta} }\right)dx\nonumber\\
\leq&\liminf\limits_{n\to\infty}\int_{\Omega_1}\left(\frac{(\mathcal{I}_\mu\ast F(u_n(x)))F(u_n(x))}{|u_n(x)|^{p\theta}}|\omega_n(x)|^{p\theta}
-  \frac{ C }{ \|u_n\|_W^{p\theta} }\right)dx\nonumber\\
= &\liminf\limits_{n\to\infty}\int_{\Omega_1}\left(\frac{(\mathcal{I}_\mu\ast F(u_n))F(u_n) }{  \|u_n\|_W^{p\theta} }
-  \frac{ C }{ \|u_n\|_W^{p\theta} }\right)dx\nonumber\\
\leq & \liminf\limits_{n\to\infty}\int_{\Omega_1} \frac{\mathscr{M}(1)(\mathcal{I}_\mu\ast F(u_n))F(u_n) }{  \mathscr{M}(\|u_n\|_W^{p}) }dx
- \liminf\limits_{n\to\infty} \int_{\Omega_1}\frac{C }{ \|u_n\|_W^{p\theta} }dx\nonumber\\
\leq & \liminf\limits_{n\to\infty}\int_{\mathbb{R}^N} \frac{\mathscr{M}(1)(\mathcal{I}_\mu\ast F(u_n))F(u_n) }{  \mathscr{M}(\|u_n\|_W^{p}) }dx
- \liminf\limits_{n\to\infty} \int_{\Omega_1}\frac{C }{ \|u_n\|_W^{p\theta} }dx\nonumber\\
= &\frac{ \mathscr{M}(1)}{p}\liminf\limits_{n\to\infty}\int_{\mathbb{R}^N} \frac{(\mathcal{I}_\mu\ast F(u_n))F(u_n) }{ \frac{1}{p} \mathscr{M}(\|u_n\|_W^{p}) }dx\nonumber\\
= &\frac{ \mathscr{M}(1)}{p}\liminf\limits_{n\to\infty} \frac{\int_{\mathbb{R}^N}(\mathcal{I}_\mu\ast F(u_n))F(u_n) dx}{ \frac{\lambda}{2}\int_{\mathbb{R}^{N}}(\mathcal{I}_\mu\ast F(u_n))F(u_n)dx+c-o(1) }.
\end{align}
So by (\ref{equ6}) and (\ref{equ6a}), we get
$$
+\infty\leq \frac{2 \mathscr{M}(1)}{p\lambda}.
$$
This is a contradiction. This shows that $\mbox{meas}(\Omega_1)=0$. Hence $\omega(x)=0$ for almost all $x\in\mathbb{R}^N$. The convergence in (\ref{equ3}) means that
\begin{align}\label{equ3a}
\omega_n(x)\to 0\quad \mbox{a.e.\ in}\ \mathbb{R}^N,\quad \mbox{and}\ \
\omega_n(x)\to0\quad \mbox{a.e.\ in}\ L^r(\mathbb{R}^N)\ \ \ \mbox{as}\ n\to\infty,
\end{align}
for $p\leq r< \frac{Np}{N-ps}$.

Using (\ref{equ1}) and $(M_2)$, we get
\begin{align}\label{equ7}
c+1\geq&\mathcal{J}_\lambda(u_n)-\frac{1}{p\theta}\langle\mathcal{J}'_\lambda(u_n),u_n\rangle\nonumber\\
=&\frac{1}{p}\mathscr{M}(\|u_n\|_W^p)-   \frac{1}{p\theta}M(\|u_n\|_W^p)\|u_n\|_W^p\nonumber\\
&+\lambda \int_{\mathbb{R}^{N}}(\mathcal{I}_\mu\ast F(u_n))\left(\frac{1}{p\theta}f(u_n)u_n-\frac{1}{2}F(u_n)\right)dx\nonumber\\
\geq&\lambda \int_{\mathbb{R}^{N}}(\mathcal{I}_\mu\ast F(u_n))\left(\frac{1}{p\theta}f(u_n)u_n-\frac{1}{2}F(u_n)\right)dx\nonumber\\
=&\lambda \int_{\mathbb{R}^{N}}(\mathcal{I}_\mu\ast F(u_n))\mathscr{F}(u_n)dx,
\end{align}
for $n$ large enough.

Let us define $\Omega_n(a,b):=\{x\in\mathbb{R}^N:a\leq |u_n(x)|\leq b\}$ for $a,b\geq 0$. From $(M_1)$ and $(M_2)$, we have that
\begin{align}\label{equ8fgh}
\mathscr{M}(\|u_n\|_W^p)\geq\frac{1}{\theta}M(\|u_n\|_W^p) \|u_n\|_W^p \geq\frac{m_0}{\theta}  \|u_n\|_W^p.
\end{align}
This together with (\ref{equ2}) and (\ref{equ5b}) yields that
\begin{align}\label{equ8}
0<\frac{2}{p\lambda}\leq&
\limsup\limits_{n\to\infty}\frac{\int_{\mathbb{R}^{N}}(\mathcal{I}_\mu\ast F(u_n))F(u_n)dx} {\mathscr{M}(\|u_n\|_W^p)}\nonumber\\
=&\limsup\limits_{n\to\infty}\int_{\mathbb{R}^N}\frac{(\mathcal{I}_\mu\ast F(u_n))F(u_n)} {\mathscr{M}(\|u_n\|_W^p)}dx\nonumber\\
=&\limsup\limits_{n\to\infty}\left(\int_{\Omega_n(0,r_0)}+\int_{\Omega_n(r_0,\infty)}\right) \frac{(\mathcal{I}_\mu\ast F(u_n))F(u_n)} {\mathscr{M}(\|u_n\|_W^p)}dx.
\end{align}
On the one hand, by Lemma \ref{estf}, (\ref{equ8fgh}),
$(F_2)$ and (\ref{equ3a}), we obtain
\begin{align}\label{equ8a}
&\int_{\Omega_n(0,r_0)} \frac{(\mathcal{I}_\mu\ast F(u_n))F(u_n)} {\mathscr{M}(\|u_n\|_W^p)}dx\nonumber\\
\leq& \frac{K\theta}{m_0}\int_{\Omega_n(0,r_0)} \frac{ |F(u_n)|} { \|u_n\|_W^p }dx\nonumber\\
\leq& \frac{c_0K\theta}{m_0}\int_{\Omega_n(0,r_0)} \left(\frac{ |u_n|^{q_1}} {q_1 \|u_n\|_W^p }+\frac{ |u_n|^{q_2}} {q_2 \|u_n\|_W^p }\right)dx\nonumber\\
=& \frac{c_0K\theta}{m_0}\int_{\Omega_n(0,r_0)} \left(\frac{ |u_n|^{q_1-p}} {q_1 }|\omega_n|^p+\frac{ |u_n|^{q_2-p}} {q_2 }|\omega_n|^p\right)dx\nonumber\\
\leq & \frac{c_0K\theta}{m_0} \left(\frac{ r_0^{q_1-p}} {q_1 } +\frac{ r_0^{q_2-p}} {q_2 }\right)\int_{\Omega_n(0,r_0)}|\omega_n|^pdx\to0,\ \ \mbox{as}\ n\to\infty.
\end{align}
On the other hand, using H\"{o}lder inequality,  (\ref{equ3a}), (\ref{equ7}) and $(F_4)$, we find
\begin{align}\label{equ8b}
& \int_{\Omega_n(r_0,\infty)}  \frac{|\mathcal{I}_\mu\ast F(u_n)|F(u_n)} {\mathscr{M}(\|u_n\|_W^p)}dx\nonumber\\
 \leq &\frac{\theta}{m_0}\int_{\Omega_n(r_0,\infty)}  \frac{|\mathcal{I}_\mu\ast F(u_n)|F(u_n)} { \|u_n\|_W^p}dx\nonumber\\
= &\frac{\theta}{m_0}\int_{\Omega_n(r_0,\infty)}  \frac{|\mathcal{I}_\mu\ast F(u_n)|F(u_n)} { |u_n|^p}|\omega_n(x)|^pdx\nonumber\\
\leq&\frac{\theta}{m_0}\left(\int_{\Omega_n(r_0,\infty)}  \left(\frac{|\mathcal{I}_\mu\ast F(u_n)|F(u_n)} { |u_n|^p}\right)^\kappa dx\right)^{\frac{1}{\kappa}}
\left(\int_{\Omega_n(r_0,\infty)}   |\omega_n(x)|^{\frac{\kappa p}{\kappa-1}}dx\right)^{\frac{\kappa-1}{\kappa}}\nonumber\\
\leq&\frac{\theta}{m_0}c_1^{\frac{1}{\kappa}}\left(\int_{\Omega_n(r_0,\infty)}  |\mathcal{I}_\mu\ast F(u_n)|^\kappa\mathscr{F}(u_n) dx\right)^{\frac{1}{\kappa}}
\left(\int_{\Omega_n(r_0,\infty)}   |\omega_n(x)|^{\frac{\kappa p}{\kappa-1}}dx\right)^{\frac{\kappa-1}{\kappa}}\nonumber\\
\leq&\frac{\theta}{m_0}c_1^{\frac{1}{\kappa}}K^{\frac{\kappa-1}{\kappa}}\left(\int_{\Omega_n(r_0,\infty)}  |\mathcal{I}_\mu\ast F(u_n)| \mathscr{F}(u_n) dx\right)^{\frac{1}{\kappa}}
\left(\int_{\Omega_n(r_0,\infty)}   |\omega_n(x)|^{\frac{\kappa p}{\kappa-1}}dx\right)^{\frac{\kappa-1}{\kappa}}\nonumber\\
\leq&\frac{\theta}{m_0}c_1^{\frac{1}{\kappa}}K^{\frac{\kappa-1}{\kappa}}\left(\frac{c+1}{\lambda}\right)^{\frac{1}{\kappa}}
\left(\int_{\Omega_n(r_0,\infty)}   |\omega_n(x)|^{\frac{\kappa p}{\kappa-1}}dx\right)^{\frac{\kappa-1}{\kappa}}\to0,\ \   \mbox{as}\ n\to\infty.
\end{align}
Here we used the fact that $\frac{\kappa p}{\kappa-1}\in (p,\frac{Np}{N-ps})$ if $\kappa>\frac{N}{ps}$.
Thus, we get a contradiction from (\ref{equ8})-(\ref{equ8b}). The proof is complete.
\end{proof}

\begin{lemma}\label{mage2}
Assume that $(V)$, $(M_1)-(M_2)$ and $(F_1)-(F_4)$ hold. Then the functional $\mathcal{J}_\lambda$ satisfies $(C)_c-$condition for any $\lambda>0$.
\end{lemma}

\begin{proof}
Suppose that  $\{u_n\}\subset W^{s,p}_V(\mathbb{R}^N)$ is a $(C)_c-$sequence for $\mathcal{J}_\lambda(u)$, from Lemma \ref{mage1}, we have that $\{u_n\}$ is bounded in $W^{s,p}_V(\mathbb{R}^N)$, then if necessary to a subsequence, we have
\begin{align}\label{cc12a}
&u_n\rightharpoonup u \ \ \mbox{in}\ W^{s,p}_V(\mathbb{R}^N),\quad u_n\to u\ \ \mbox{a.e.\ in}\ \mathbb{R}^N,\nonumber\\
&u_n\to u\ \ \mbox{in} \ L^{q_1}(\mathbb{R}^N)\cap L^{q_2}(\mathbb{R}^N),\\
&|u_n|\leq\ell(x)\ \ \mbox{a.e.\ in}\ \mathbb{R}^N,\ \ \mbox{for\ some}\ \ell(x)\in L^{q_1}(\mathbb{R}^N)\cap L^{q_2}(\mathbb{R}^N).\nonumber
\end{align}
For simplicity, let $\varphi\in W_V^{s,p}(\mathbb{R}^N)$ be fixed and denote by $B_\varphi$ the linear functional on
$W_V^{s,p}(\mathbb{R}^N)$ defined by
\begin{align*}
B_\varphi(v)&=\iint_{\mathbb{R}^{2N}}\frac{|\varphi(x)-\varphi(y)|^{p-2}(\varphi(x)-\varphi(y))
}{|x-y|^{N+ps}}(v(x)-v(y))\,dx\,dy.
\end{align*}
for all $v\in W_V^{s,p}(\mathbb{R}^N)$.
By H\"older inequality, we have
\[
\left|B_\varphi(v)\right|\leq
[\varphi]_{s,p}^{p-1}[v]_{s,p}  \leq
\|\varphi\|_W^{p-1}\|v\|_W,
\]
for all $v\in W_V^{s,p}(\mathbb{R}^N)$.
Hence, (\ref{cc12a}) gives that
\begin{align}\label{cc12b}
\lim\limits_{n\to\infty}\Big(M(\|u_n\|_W^p)-M(\|u\|_W^p)\Big)B_u(u_n-u)=0,
\end{align}
since $\Big\{M(\|u_n\|_W^p)-M(\|u\|_W^p)\Big\}_n$ is bounded in $\mathbb{R}$.

Since $\mathcal{J}'_\lambda(u_n)\to0$ in $(W^{s,p}_V(\mathbb{R}^N))'$ and $u_n\rightharpoonup u$ in  $W^{s,p}_V(\mathbb{R}^N)$, we have
$$
\langle \mathcal{J}'_\lambda(u_n)-\mathcal{J}'_\lambda(u),u_n-u\rangle\to0\ \ \ \mbox{as}\ n\to\infty.
$$
That is,
\begin{align}\label{erq3}
o(1)=&\langle \mathcal{J}'_\lambda(u_n)-\mathcal{J}'_\lambda(u),u_n-u\rangle\nonumber\\
=&M(\|u_n\|_W^p)\Big(B_{u_n}(u_n-u)+\int_{\mathbb{R}^N}V(x)|u_n|^{p-2}u_n(u_n-u)dx\Big)\nonumber\\
&-M(\|u\|_W^p)\Big(B_{u}(u_n-u)+\int_{\mathbb{R}^N}V(x)|u|^{p-2}u(u_n-u)dx\Big)\nonumber\\
&-\lambda\int_{\mathbb{R}^N}\Big[(\mathcal{I}_\mu\ast F(u_n))f(u_n)-(\mathcal{I}_\mu\ast F(u))f(u)\Big](u_n-u)dx\nonumber\\
=&M(\|u_n\|_W^p) \Big[B_{u_n}(u_n-u)- B_{u}(u_n-u)\Big] \nonumber\\
&+\Big(M(\|u_n\|_W^p)-M(\|u\|_W^p)\Big) B_{u}(u_n-u)\nonumber\\
&+M(\|u_n\|_W^p) \int_{\mathbb{R}^N}V(x)(|u_n|^{p-2}u_n-|u|^{p-2}u)(u_n-u)dx\nonumber\\
&+[M(\|u_n\|_W^p)-M(\|u\|_W^p)] \int_{\mathbb{R}^N}V(x)|u|^{p-2}u(u_n-u)dx\nonumber\\
&-\lambda\int_{\mathbb{R}^N}\Big[(\mathcal{I}_\mu\ast F(u_n))f(u_n)-(\mathcal{I}_\mu\ast F(u))f(u)\Big](u_n-u)dx.
\end{align}
From Lemma \ref{lema2a}, we have
\begin{align}\label{cc12}
\int_{\mathbb{R}^N}\Big[(\mathcal{I}_\mu\ast F(u_n))f(u_n)-(\mathcal{I}_\mu\ast F(u))f(u)\Big](u_n-u)dx\to0,\quad \mbox{as}\ n\to\infty.
\end{align}
Moreover, using H\"{o}lder inequality and (\ref{cc12a}), we have
\begin{align}\label{erq3yu}
&[M(\|u_n\|_W^p)-M(\|u\|_W^p)] \int_{\mathbb{R}^N}V(x)|u|^{p-2}u(u_n-u)dx\to0,\quad \mbox{as}\ n\to\infty.
\end{align}
From (\ref{cc12b})-(\ref{erq3yu}) and $(M_1)$, we obtain
\begin{align*}
\lim\limits_{n\to\infty}M(\|u_n\|_W^p)\left(\Big[B_{u_n}(u_n-u)- B_{u}(u_n-u)\Big]+\int_{\mathbb{R}^N}V(x)(|u_n|^{p-2}u_n-|u|^{p-2}u)(u_n-u)dx\right)=0.
\end{align*}
Since $M(\|u_n\|_W^p)[ B_{u_n}(u_n-u)- B_{u}(u_n-u)] \geq0$ and $V(x)(|u_n|^{p-2}u_n-|u|^{p-2}u)(u_n-u)\geq0$ for all $n$ by convexity, $(M_1)$ and $(V_1)$,
we have
\begin{align}\label{minti1}
\lim\limits_{n\to\infty} \Big[B_{u_n}(u_n-u)- B_{u}(u_n-u)\Big]=0,\nonumber\\
\lim\limits_{n\to\infty} \int_{\mathbb{R}^N}V(x)(|u_n|^{p-2}u_n-|u|^{p-2}u)(u_n-u)dx=0.
\end{align}
Let us now recall the well-known Simon inequalities. There
exist positive numbers $c_p$ and $C_p$, depending only
on $p$, such that
\begin{equation} \label{simon}
|\xi-\eta|^{p}
\leq  \begin{cases}
 c_p (|\xi|^{p-2}\xi-|\eta|^{p-2}\eta)(\xi-\eta)
& \text{for } p\geq2,\\[3pt]
C_p\big[(|\xi|^{p-2}\xi-|\eta|^{p-2}\eta)  (\xi-\eta) \big]^{p/2}(|\xi|^p+|\eta|^p)^{(2-p)/2}
& \text{for }1<p<2,
\end{cases}
\end{equation}
for all $\xi,\eta\in\mathbb{R}^N$. According to the Simon
inequality, we divide the discussion into two cases.

{\it Case $p\geq 2$:} From (\ref{minti1}) and (\ref{simon}), as $n\to\infty$, we have
\begin{align*}
[u_n-u]_{s,p}^p=&\iint_{\mathbb{R}^{2N}}\frac{|u_n(x)-u(x)-u_n(y)+u(y)|^p}{|x-y|^{N+ps}}dxdy\\
\leq &c_p\iint_{\mathbb{R}^{2N}}\frac{|u_n(x) -u_n(y)|^{p-2}(u_n(x) -u_n(y))-|u(x) -u(y)|^{p-2}(u(x) -u(y))}{|x-y|^{N+ps}}\nonumber\\
&\qquad \times \Big(u_n(x)-u(x)-u_n(y)+u(y)\Big)dxdy\\
=&c_p\Big[B_{u_n}(u_n-u)- B_{u}(u_n-u)\Big]=o(1),
\end{align*}
and
\begin{align*}
\|u_n-u\|_{p,V}^p\leq c_p\int_{\mathbb{R}^N}V(x)(|u_n|^{p-2}u_n-|u|^{p-2}u)(u_n-u)dx=o(1).
\end{align*}
Consequently, $\|u_n-u\|_W\to0$ as $n\to\infty$.

{\it Case $1<p<2$:} taking $\xi=u_n(x)-u_n(y)$ and $\eta=u(x)-u(y)$ in  (\ref{simon}), as $n\to\infty$, we have
\begin{align*}
[u_n-u]_{s,p}^p
\leq& C_p\big[B_{u_n}(u_n-u)-B_{u}(u_n-u) \big]^{p/2}([u_n]_{s,p}^p+[u]_{s,p}^p)^{(2-p)/2}\\
\leq& C_p\big[B_{u_n}(u_n-u)-B_{u}(u_n-u) \big]^{p/2}([u_n]_{s,p}^{p(2-p)/2}+[u]_{s,p}^{p(2-p)/2})\\
\leq& C\big[B_{u_n}(u_n-u)-B_{u}(u_n-u) \big]^{p/2}=o(1).
\end{align*}
Here we used the fact that $[u_n]_{s,p}$ and $[u]_{s,p}$ are bounded, and the elementary inequality
$$
(a+b)^{(2-p)/2}\leq a^{(2-p)/2}+b^{(2-p)/2}\ \ \mbox{for\ all}\ a,b\geq 0\ \mbox{and}\ 1<p<2.
$$
Moreover, by H\"{o}lder inequality and (\ref{minti1}), as $n\to\infty$,
\begin{align*}
\|u_n-u\|_{p,V}^p\leq &C_p\int_{\mathbb{R}^N}V(x) \big[(|u_n|^{p-2}u_n-|u|^{p-2}u)  (u_n-u) \big]^{p/2}(|u_n|^p+|u|^p)^{(2-p)/2}dx\\
\leq&C_p\left(\int_{\mathbb{R}^N}V(x)  (|u_n|^{p-2}u_n-|u|^{p-2}u)  (u_n-u) dx\right)^{p/2}\\
&\qquad \times\left(\int_{\mathbb{R}^N}V(x)(|u_n|^p+|u|^p)dx\right)^{(2-p)/2}\\
\leq&C_p\left(\|u_n\|_{p,V}^{p(2-p)/2}+\|u\|_{p,V}^{p(2-p)/2}\right)\left(\int_{\mathbb{R}^N}V(x)  (|u_n|^{p-2}u_n-|u|^{p-2}u)  (u_n-u) dx\right)^{p/2}\\
\leq&C \left(\int_{\mathbb{R}^N}V(x)  (|u_n|^{p-2}u_n-|u|^{p-2}u)  (u_n-u) dx\right)^{p/2}\to0.
\end{align*}
Thus $\|u_n-u\|_W\to0$ as $n\to\infty$. The proof is complete.
\end{proof}

Now we are ready to prove our main result.

\noindent
{\it Proof of Theorem \ref{thm1.1}:} By Lemmas \ref{mage}-\ref{mage2} and using Lemma \ref{cmainth}, we obtain that there exists a critical point of functional $\mathcal{J}_\lambda$, so problem \eqref{e1.1} has a nontrivial weak solution for any $\lambda>0$.

\subsection*{Acknowledgments}

The author was supported by National Nature Science Foundation of China (No. 11501468) and Chongqing Research Program of Basic Research and Frontier Technology cstc2018jcyjAX0196.

\end{document}